\documentclass[11pt]{article}
\usepackage[T1]{fontenc}
\usepackage{lmodern}
\usepackage{microtype}
\usepackage[a4paper,margin=1in]{geometry}
\usepackage{amsmath,amssymb,amsthm,mathtools}
\usepackage{aliascnt}
\usepackage{booktabs,array,enumitem,needspace}
\usepackage[hidelinks]{hyperref}
\usepackage[nameinlink,noabbrev,capitalise]{cleveref}
\hypersetup{pdftitle={When Do Restricted Sumsets Cover a Finite Abelian Group?},
 pdfsubject={Restricted sumset covering in finite abelian groups}}
\newcommand{\proofcase}[1]{\par\smallskip\Needspace{3\baselineskip}\noindent\textit{#1.}\ }

\newtheorem{theorem}{Theorem}[section]
\newaliascnt{lemma}{theorem}
\newtheorem{lemma}[lemma]{Lemma}
\aliascntresetthe{lemma}
\newaliascnt{proposition}{theorem}
\newtheorem{proposition}[proposition]{Proposition}
\aliascntresetthe{proposition}
\newaliascnt{corollary}{theorem}
\newtheorem{corollary}[corollary]{Corollary}
\aliascntresetthe{corollary}
\theoremstyle{remark}
\newaliascnt{remark}{theorem}
\newtheorem{remark}[remark]{Remark}
\aliascntresetthe{remark}

\crefname{theorem}{Theorem}{Theorems}
\crefname{lemma}{Lemma}{Lemmas}
\crefname{proposition}{Proposition}{Propositions}
\crefname{corollary}{Corollary}{Corollaries}
\crefname{remark}{Remark}{Remarks}
\crefname{section}{Section}{Sections}
\Crefname{theorem}{Theorem}{Theorems}
\Crefname{lemma}{Lemma}{Lemmas}
\Crefname{proposition}{Proposition}{Propositions}
\Crefname{corollary}{Corollary}{Corollaries}
\Crefname{remark}{Remark}{Remarks}
\Crefname{section}{Section}{Sections}

\newcommand{\Stab}{\operatorname{Stab}}
\newcommand{\supp}{\operatorname{supp}}
\newcommand{\eps}{\varepsilon}
\newcommand{\ceil}[1]{\left\lceil #1\right\rceil}
\newcommand{\floor}[1]{\left\lfloor #1\right\rfloor}
\newcommand{\F}{\mathbb F}
\newcommand{\Z}{\mathbb Z}

\title{When Do Restricted Sumsets Cover a Finite Abelian Group?}
\author{Dongchun Han\thanks{School of Mathematics, Southwest Jiaotong University, Chengdu 610031, P.R. China.}
\and Yuan Ren\thanks{School of Mathematical Sciences, University of Electronic Science and Technology of China, Chengdu 611731, P.R. China. Corresponding author: \texttt{rytgyx@126.com}.}}
\date{}

\begin{document}
\maketitle
\vspace{-2.2em}

\begin{abstract}
For a subset $A$ of a finite abelian group $G$, let $\Sigma_k(A)$ denote the set of sums of $k$ distinct elements of $A$, and let $\Sigma(A)=\bigcup_{k\ge1}\Sigma_k(A)$. The problem of determining when $\Sigma(A)=G$ goes back to Erd\H{o}s and Heilbronn. Nguyen and Vu later formulated the corresponding fixed-length question: for a prescribed $k$, when does $\Sigma_k(A)=G$? For groups of odd order, we prove the sharp large-order threshold $|A|>|G|/3+1$, which forces $\Sigma_k(A)=G$ throughout the entire middle range $4\le k\le |A|-4$. This determines the limiting thresholds posed by Girard, Griffiths and Hamidoune. If $p$ is the least prime divisor of $|G|$, then, for fixed odd $p$ and sufficiently large $|G|$, the exact threshold $|A|\ge |G|/p+p-1$ forces covering throughout the middle range beginning at $2p-2$. Thus, throughout this range, prescribing the number of summands does not change the sharp covering threshold. For $3\le k\le2p-3$, the restricted and ordinary $k$-fold covering problems have the same asymptotic threshold. As an application to MDS elliptic codes on an elliptic curve $E/\F_q$, we sharpen the half-scale conjecture of Li, Wan and Zhang, obtaining the sharp $1/3$-scale bound when $|E(\F_q)|$ is odd and, more generally, the exact bound determined by the least prime divisor of $|E(\F_q)|$.
\end{abstract}



\section{Introduction}\label{sec:introduction}

Throughout the paper $G$ is a finite abelian group of order $N$, written additively. For $A\subset G$ and $0\le k\le |A|$, put $\Sigma_k(A)=\{a_1+\cdots+a_k:a_1,\ldots,a_k\in A\text{ pairwise distinct}\}$, and write $\Sigma(A)=\bigcup_{k\ge1}\Sigma_k(A)$. The problem of determining when $\Sigma(A)=G$ goes back to Erd\H{o}s and Heilbronn. Nguyen and Vu later formulated two closely related versions of the problem in a unified way \cite[Questions~1.1 and~1.2]{NguyenVu}:
\[
\text{When does }\Sigma(A)=G?
\]
and, for a prescribed integer $k$,
\[
\text{When does }\Sigma_k(A)=G?
\]

The first question is classical. Following work of Erd\H{o}s and Heilbronn, Diderrich and Mann, Gao and Hamidoune, and others \cite{DiderrichMann,ErdosHeilbronn,GaoHamidoune}, the sharp threshold was finally determined for all finite abelian groups by Freeze, Gao and Geroldinger \cite{FreezeGaoGeroldinger,FreezeGaoGeroldingerCorrigendum}. If $p$ is the least prime divisor of $N$, then, for fixed $p$ and sufficiently large $N$, the sharp threshold is $|A|\ge N/p+p-1$. This leads naturally to the fixed-length problem: for a prescribed $k$, how large must $A$ be to ensure $\Sigma_k(A)=G$?

The fixed-length problem is a natural question in restricted addition. The study of lower bounds for restricted sumsets goes back to the Erd\H{o}s--Heilbronn problem \cite{ErdosHeilbronn}. Dias da Silva and Hamidoune proved the sharp prime-cyclic estimate
\[
|\Sigma_k(A)|\ge \min\{p,k(|A|-k)+1\},
\]
settling the general Erd\H{o}s--Heilbronn conjecture \cite{DiasHamidoune}; see also \cite[Theorem~3.3]{ANR} for this formulation. Alon, Nathanson and Ruzsa later gave a polynomial-method proof and further extensions \cite[Theorems~3.2 and~3.3]{ANR}. For general finite abelian groups, Lev developed a systematic theory of restricted addition \cite[Theorem~3]{LevI}, while Hamidoune, Llad\'o and Serra obtained further expansion and structural results \cite[p.~513]{HLS}.

For fixed-length covering, several results are known in specific cases. Gallardo, Grekos, Habsieger, Hennecart, Landreau and Plagne studied two- and three-fold restricted sums in $\Z/n\Z$. In particular, they proved three-fold covering under the condition $|A|>n/2$ and conjectured that, for odd $n$, this could be lowered to $2n/5+O(1)$ \cite[Theorem~3.1 and Conjecture~3.9]{GallardoEtAl}. Lev subsequently resolved the three-fold covering conjecture for arbitrary finite abelian groups \cite[Main Theorem]{LevIII}.

For groups of even order, Roth and Lempel showed that, apart from the exceptional endpoint cases, $|A|>|G|/2$ implies $\Sigma_k(A)=G$ throughout the middle range \cite[Theorem~3.1 and Remark~3.1]{RothLempel}.

The odd-order problem remained much less clear. Girard, Griffiths and Hamidoune observed that, in an odd-order group, $|A|\ge (N+3)/2$ implies $\Sigma_2(A)=G$, and they expected the same hypothesis to imply $\Sigma_k(A)=G$ for every $2\le k\le |A|-2$ in every sufficiently large odd-order group \cite[Section~6]{GGH}. More generally, they introduced
\[
c_k(n)=\inf\{c: |A|>c|G|\Longrightarrow \Sigma_k(A)=G\text{ for every }A\subset G\text{ with }|G|=n\}
\]
and
\[
c_k=\limsup_{m\to\infty}c_k(2m+1),
\]
and asked for the value of $c_k$ for each fixed $k$ \cite[Problem~6.1]{GGH}.

Tang and Wei proved that $|A|>0.4045|G|$ implies $\Sigma_4(A)=G$ for sufficiently large odd cyclic groups \cite[Theorem~1.3]{TangWei}, while Goswami and Mistri obtained thresholds tending to $1/3$ as the restricted length grows \cite[Theorem~1.8]{GoswamiMistri}. More recently, Chen and Huang proved covering results for odd-order finite abelian groups in the range $3\le k\le N/p-2$, where $p$ is the least prime divisor of $N$, with set-size bounds approximately $2N/5$ or $5N/13$ \cite[Theorem~B]{ChenHuang}. Their results leave open both the sharp threshold and the remaining middle lengths.

\begin{theorem}\label{thm:sharp-density}
Let $G$ be a finite abelian group of odd order $N>605$, and let $A\subset G$. If $|A|>N/3+1$, then
\[
\Sigma_k(A)=G\qquad (4\le k\le |A|-4).
\]
\end{theorem}

Theorem~\ref{thm:sharp-density} strengthens the odd-order covering expectation of Girard, Griffiths and Hamidoune: their half-scale hypothesis can be replaced by the sharp bound $|A|>N/3+1$ for $4\le k\le |A|-4$, and the index-three construction below shows that this threshold is best possible. Together with the known two-fold result and Lev's three-fold theorem \cite[Main Theorem]{LevIII}, this confirms their sufficiently-large odd-order covering expectation in full. The limiting thresholds at lengths two and three are $1/2$ and $2/5$, respectively. Thus these results determine all the limiting thresholds posed by Girard, Griffiths and Hamidoune.

\begin{corollary}\label{cor:critical-densities}
For finite abelian groups of odd order,
\[
c_2=\frac12,\qquad c_3=\frac25,\qquad c_k=\frac13\quad(k\ge4).
\]
\end{corollary}

Suppose $N$ is composite. Let $p$ be the least prime divisor of $N$, let $H<G$ have index $p$, choose $y\notin H$, and take $D\subset y+H$ with $|D|=p-2$. Then $A=H\cup D$ has size $N/p+p-2$, but the coset $(p-1)y+H$ is disjoint from every restricted sumset $\Sigma_k(A)$. For $p=3$, this is exactly the obstruction showing that the threshold in \cref{thm:sharp-density} is sharp. More generally, it is the extremal construction underlying the classical subset-sum threshold $N/p+p-1$. The same threshold remains sharp when the number of summands is prescribed.

\begin{theorem}\label{thm:least-prime}
Fix an odd prime $p$. There exists $N_0=N_0(p)$ such that the following holds. Let $G$ be a finite abelian group of order $N\ge N_0$ whose least prime divisor is $p$. If $A\subseteq G$ satisfies $|A|\ge N/p+p-1$, then
\[
\Sigma_k(A)=G\qquad(2p-2\le k\le |A|-2p+2).
\]
\end{theorem}

Thus, throughout this range, prescribing the number of summands does not change the sharp covering threshold.

The starting length $2p-2$ is also essentially best possible. If $q=2p-1$ is prime and $t\ge2$, take $G=(C_p)^t\oplus C_q$ and $A=(C_p)^t\times\{0,1\}$. Then $p$ is the least prime divisor of $N$ and $|A|>N/p+p-1$, but $\Sigma_{2p-3}(A)\ne G$. More generally, the prime-gap theorem of Baker, Harman and Pintz \cite[Theorem~1]{BHP} gives examples with failure at a length $(2-o(1))p$. Hence the leading coefficient $2$ in the starting length is asymptotically optimal.

For shorter prescribed lengths, the threshold behaves differently. For integers $N\ge2$ and $k\ge1$, put
\[
\alpha_k(N)=\max_{\substack{d\mid N\\ d>1}}
\frac{\ceil{(d-1)/k}}{d}.
\]
Bajnok \cite[Theorem~9]{BajnokCritical} proved that the ordinary $k$-fold critical number is $\alpha_k(N)N+1$. Thus $\alpha_k(N)N$ is the largest possible size of a set $B\subset G$ with $kB\ne G$. Since $\Sigma_k(B)\subseteq kB$, the same quantity is an immediate lower bound for the restricted problem. For shorter prescribed lengths, this lower bound is asymptotically sharp.

\begin{theorem}\label{thm:asymptotic-critical}
Let $G$ be a finite abelian group of order $N$, and let $p$ be the least prime divisor of $N$. Let $3\le k\le2p-3$ and $\eps>0$. For all sufficiently large $N$, depending on $k$ and $\eps$, every $A\subseteq G$ with $|A|\ge(\alpha_k(N)+\eps)N$ satisfies $\Sigma_k(A)=G$.
\end{theorem}

For example, when $N=q$ is prime, one has $\alpha_k(q)=\frac{\lceil(q-1)/k\rceil}{q}=\frac1k+O(q^{-1})$ for each fixed $k$. Thus, for shorter lengths, the threshold depends on $k$ and on the divisor structure of $N$, while for $2p-2\le k\le |A|-2p+2$, Theorem~\ref{thm:least-prime} gives the exact threshold $N/p+p-1$. Together with Theorem~\ref{thm:sharp-density}, these results give the limiting threshold $1/3$ for every fixed $k\ge4$ in odd-order groups. The short-length theorem also determines the first covering length when $|A|\ge c|G|$ for a fixed $c>0$.

\begin{corollary}\label{cor:positive-density}
Let $0<c<1$. There exists $p_0(c)$ such that if the least prime divisor of $|G|$ is larger than $p_0(c)$ and $A\subset G$ satisfies $|A|\ge c|G|$, then
\[
\Sigma_k(A)=G\qquad\left(\frac1c<k<|A|-\frac1c\right).
\]
The lower endpoint is best possible.
\end{corollary}

There is also a direct application to MDS elliptic codes. Let $E/\F_q$ be an elliptic curve. Since $E(\F_q)$ is a finite abelian group, the MDS condition for an elliptic evaluation code can be expressed as the failure of a prescribed group element to lie in a fixed-length restricted sumset of the evaluation set; see \cref{lem:MDS-criterion}. Covering results for $E(\F_q)$ therefore translate directly into upper bounds for the lengths of such codes.

Li, Wan and Zhang conjectured a one-half maximal-length scale for MDS elliptic codes \cite[p.~2]{LiWanZhang}. Han and Ren obtained an asymptotic answer and bounds within an additive constant of the half-group scale \cite[Theorem~1.2]{HanRen}, while Chen and Huang proved the exact half-group bound for $q\ge47089$ \cite[Theorem~32]{ChenHuang}. Our results give the sharp $1/3$-scale bound when $|E(\F_q)|$ is odd and, more generally, the exact bound determined by the least prime divisor of $|E(\F_q)|$.

At the three-fold endpoints, \cref{lem:MDS-criterion} together with \cref{prop:three-fold}(a) also gives $n\le2|E(\F_q)|/5$ when $|E(\F_q)|$ is odd and at least $149$, and $k\in\{3,n-3\}$.

\begin{theorem}\label{thm:elliptic}
Let $C$ be an $[n,k]$ MDS elliptic code on an elliptic curve $E/\F_q$.
\begin{enumerate}[label=\textup{(\alph*)},leftmargin=2.5em]
\item Suppose $|E(\F_q)|$ is odd. If $q>655$ and $4\le k\le n-4$, then
\[
n\le \frac{|E(\F_q)|}{3}+1.
\]
\item Suppose the least prime divisor $p$ of $|E(\F_q)|$ is odd. For fixed $p$ and sufficiently large $q$, if $2p-2\le k\le n-2p+2$, then
\[
n\le \frac{|E(\F_q)|}{p}+p-2.
\]
\end{enumerate}
\end{theorem}

When $|E(\F_q)|$ is even, the corresponding half-scale covering result follows from Roth and Lempel \cite[Theorem~3.1 and Remark~3.1]{RothLempel}. Combining the even- and odd-order cases with the Hasse--Weil bound gives the following uniform consequence.

\begin{corollary}\label{cor:elliptic-half}
Let $C$ be an $[n,k]$ MDS elliptic code on $E/\F_q$, with $3\le k\le n-3$. If $q>655$, then $n\le |E(\F_q)|/2$.
\end{corollary}

The proofs use two main ideas. For the one-third theorem, restricted-sum estimates and Kneser's theorem are used to show that suitable sums of fibres fill the required cosets, while Lemma~\ref{lem:quotient-selection} chooses distinct unused elements whose images have any prescribed sum in the quotient. For the exact least-prime threshold, if some group element has few representations, arithmetic removal and Kneser's theorem force almost all of $A$ into one coset of an index-$p$ subgroup. The one-third theorem then gives the required covering inside the subgroup, while the elements outside it are used to choose the required quotient class. If every group element has many representations, Lemma~\ref{lem:counting-cover} produces a subset $T\subseteq A$ with $|T|\le |A|/2$ and $\Sigma_{2p-2}(T)=G$, from which all longer lengths in the required range follow.

The paper is organized as follows. Section~2 collects the basic tools, and Section~3 treats the small restricted sums needed later. Section~4 proves the sharp one-third theorem, while Section~5 proves the exact least-prime threshold. Section~6 deals with shorter prescribed lengths, and Section~7 gives the application to MDS elliptic codes.

\noindent\textbf{Statement of AI.} The paper was written by the authors. ChatGPT was used to assist with routine calculations and checks, language editing, and reorganization of the manuscript.

\noindent\textbf{Acknowledgment.} This work was supported by the National Natural Science Foundation of China (Grant No.~12571578).

\section{Preliminaries}\label{sec:preliminaries}
Write $C_n$ for the cyclic group of order $n$.  When $|G|>1$, write $p(G)$
for the least prime divisor of $|G|$.  Write
$G[2]=\{x\in G:2x=0\}$ for the two-torsion subgroup.
We use $\Sigma_0(A)=\{0\}$ and $\Sigma_k(A)=\varnothing$ when $k>|A|$.
For a finite set $A\subset G$, write $\sigma_A=\sum_{a\in A}a$.  For $X\subset G$, write $\langle X\rangle$ for the subgroup generated by $X$.

\Needspace{7\baselineskip}
\begin{lemma}\label{lem:symmetry}
Let $A$ be a finite subset of an abelian group $G$.
\begin{enumerate}[label=\textup{(\roman*)},leftmargin=2.5em]
\item For every $0\le k\le |A|$, one has $\Sigma_k(A)=\sigma_A-\Sigma_{|A|-k}(A)$. Consequently, to prove $\Sigma_k(A)=G$ for all $u\le k\le |A|-u$, it is enough to prove it for $u\le k\le |A|/2$.
\item Let $n\ge2u$, and suppose that every $n$-element subset $B\subset G$ satisfies $\Sigma_j(B)=G$ for $u\le j\le n-u$. Then every $A\subset G$ with $|A|\ge n$ satisfies $\Sigma_k(A)=G$ for $u\le k\le |A|-u$.
\end{enumerate}
\end{lemma}
\begin{proof}
Part~\textup{(i)} follows by taking complements of the chosen subsets in $A$.
For part~\textup{(ii)}, put $m=|A|$. If $k\le n-u$, choose any $n$-element subset $B\subset A$; then $\Sigma_k(B)=G$, and hence $\Sigma_k(A)=G$. If $k>n-u$, choose $F\subset A$ with $|F|=k-(n-u)$. Since $k\le m-u$, we have $|A\setminus F|=m-k+n-u\ge n$. Choose an $n$-element subset $B\subset A\setminus F$. Then
\[
G=\sigma_F+\Sigma_{n-u}(B)\subseteq\Sigma_k(A),
\]
which proves the claim.
\end{proof}

\begin{lemma}\label{lem:translation}
Let $u\in G$ and $A+u=\{a+u:a\in A\}$. Then, for every $k$, one has $\Sigma_k(A+u)=\Sigma_k(A)+ku$. Consequently $\Sigma_k(A)=G$ if and only if $\Sigma_k(A+u)=G$.
\end{lemma}
\begin{proof}
A sum of $k$ distinct elements of $A+u$ is obtained uniquely from a sum of $k$ distinct elements of $A$ by adding $ku$.
\end{proof}

For $S\subset G$, write $\Stab(S)=\{g\in G:g+S=S\}$.
We use the iterated stabilizer form of Kneser's theorem; see Kneser \cite{Kneser} and the two-set form in Tao--Vu \cite[Theorem~5.5, p.~200]{TaoVu}.

\begin{theorem}[Kneser]\label{thm:kneser}
Let $X_1,\ldots,X_t$ be finite non-empty subsets of an abelian group $G$, and let $H=\Stab(X_1+\cdots+X_t)$. Then
\[
|X_1+\cdots+X_t|\ge |X_1+H|+\cdots+|X_t+H|-(t-1)|H|.
\]
\end{theorem}

We also use Cauchy--Davenport in $C_p$, in the form stated in Tao--Vu \cite[Theorem~5.4, p.~200]{TaoVu}.

\begin{theorem}[Cauchy--Davenport]\label{thm:CD}
If $p$ is prime and $X_1,\ldots,X_t\subset C_p$ are non-empty, then
\[
|X_1+\cdots+X_t|\ge \min\{p,|X_1|+\cdots+|X_t|-(t-1)\}.
\]
\end{theorem}

The following lemma is a key tool in our covering arguments.

\begin{lemma}\label{lem:quotient-selection}
Let $p$ be prime, let $\pi:G\to C_p$ be a quotient, and let $M\subset G$.
For $i\in C_p$, put $r_i=|M\cap\pi^{-1}(i)|$.  Given $\ell\ge1$, set
$S_\ell=\sum_{i\in C_p}\min\{r_i,\ell\}$.  If $S_\ell\ge p+\ell-1$, then
$\pi(\Sigma_\ell(M))=C_p$.
\end{lemma}
\begin{proof}
Regard the elements of $M$ as labelled copies of their images in $C_p$.
For each residue $i$, retain $\min\{r_i,\ell\}$ copies.
Distribute the retained copies, one residue class at a time, cyclically among
$\ell$ boxes $X_1,\ldots,X_\ell$, using one continuous cycle through the boxes.
Since no residue is retained more than $\ell$ times, copies of the same
residue go to distinct boxes, and the box sizes differ by at most one.

After forgetting the labels, each $X_j$ is a subset of $C_p$, and
\[
\sum_{j=1}^{\ell}|X_j|
=\sum_{i\in C_p}\min\{r_i,\ell\}
=S_\ell\ge p+\ell-1.
\]
In particular, every $X_j$ is non-empty.  By \cref{thm:CD},
$X_1+\cdots+X_\ell=C_p$.
Since the boxes contain disjoint labelled elements of $M$, choosing one
element from each $X_j$ lifts to $\ell$ distinct elements of $M$.
Therefore
\[
C_p=X_1+\cdots+X_\ell
\subseteq\pi(\Sigma_\ell(M))\subseteq C_p,
\]
and hence $\pi(\Sigma_\ell(M))=C_p$.
\end{proof}

\Needspace{20\baselineskip}
\paragraph{A round-robin example.}
Take $p=7$ and $\ell=4$, and suppose that the fibre multiplicities over
$\bar0,\ldots,\bar5$ are $1,5,6,2,3,3$, with the remaining fibre empty.
After truncation at level $\ell$, the retained multiplicities are
$1,4,4,2,3,3$.  Starting at $X_1$, the distribution is
\[
\begin{array}{c@{\quad}c@{\quad}l@{\quad}c}
\toprule
\text{residue}&\text{copies}&\text{boxes receiving the copies}&\text{next box}\\
\midrule
\bar0&1&X_1&X_2\\
\bar1&4&X_2,X_3,X_4,X_1&X_2\\
\bar2&4&X_2,X_3,X_4,X_1&X_2\\
\bar3&2&X_2,X_3&X_4\\
\bar4&3&X_4,X_1,X_2&X_3\\
\bar5&3&X_3,X_4,X_1&X_2\\
\bottomrule
\end{array}
\]
Thus
\begin{align*}
X_1&=\{\bar0,\bar1,\bar2,\bar4,\bar5\},&
X_2&=\{\bar1,\bar2,\bar3,\bar4\},\\
X_3&=\{\bar1,\bar2,\bar3,\bar5\},&
X_4&=\{\bar1,\bar2,\bar4,\bar5\}.
\end{align*}

We shall use one two-fold covering fact directly from Gallardo et al.
\cite[Lemma~2.2]{GallardoEtAl}: if $K$ has odd order and $X\subset K$
satisfies $|X|>(|K|+1)/2$, then $\Sigma_2(X)=K$.  Indeed, for every $z\in K$,
$|X\cap(z-X)|\ge2|X|-|K|>1$, while $2x=z$ has a unique solution
in $K$.  Thus the intersection contains $x\ne z-x$, giving a restricted
pair of sum $z$.  This is precisely the odd-order specialization of their lemma.

The following prime-field restricted-sum theorem is due to Dias da Silva and Hamidoune \cite{DiasHamidoune}; we use the formulation in Alon--Nathanson--Ruzsa \cite[Theorem~3.3]{ANR}.

\begin{theorem}[Dias da Silva--Hamidoune]\label{thm:DSH}
Let $p$ be prime and let $A\subset C_p$. Then
\[
|\Sigma_k(A)|\ge\min\{p,k(|A|-k)+1\}
\]
for every $1\le k\le |A|$.
\end{theorem}

We shall use two theorems of Lev. The first is the two-fold restricted-sum estimate from \cite[Theorem~3]{LevI}.

\begin{theorem}[Lev's two-fold estimate]\label{thm:lev-two}
Let $A$ be a finite subset of an abelian group $G$. If $\Sigma_2(A)\ne A+A$, then $|\Sigma_2(A)|>(1+\sqrt5)|A|/2-(|G[2]|+2)$.
\end{theorem}

The second is Lev's three-fold restricted addition theorem \cite[Main Theorem]{LevIII}.

\begin{theorem}[Lev's three-fold theorem]\label{thm:lev-three}
Let $G$ be a finite abelian group and let $A\subset G$. If $|A|>\max\{5|G|/13,120|G[2]|+355\}$, then either $\Sigma_3(A)=G$, or $A$ is contained in a coset of an index-two subgroup, or $A$ is contained in a union of two cosets of an index-five subgroup.
\end{theorem}

Finally we use the arithmetic removal lemma for finite abelian groups. Green's work \cite{Green} gives a removal lemma for single equations in abelian groups, and the form below follows from Kr\'al--Serra--Vena \cite[Theorem~1]{KSV}. Their extra condition is automatically satisfied for the equation $x_1+\cdots+x_s=g$.

\begin{lemma}[Arithmetic removal]\label{lem:removal}
Fix $s\ge3$ and $\eta>0$. There exists $\rho_s(\eta)>0$ such that the following holds for every finite abelian group $G$ of order $N$. If the equation $x_1+\cdots+x_s=g$, with $x_i\in X_i$, has at most $\rho_s(\eta)N^{s-1}$ solutions in $X_1\times\cdots\times X_s$, then one can delete at most $\eta N$ elements from each $X_i$ so that no solution remains.
\end{lemma}

\begin{theorem}[Roth--Lempel]\label{thm:roth-lempel}
Let $G$ be a finite abelian group of even order $N\ge12$, and let $A\subset G$ with $m=|A|>N/2$.
\begin{enumerate}[label=\textup{(\alph*)},leftmargin=2.5em]
\item If $G$ is not isomorphic to $C_2^r$ or $C_4\oplus C_2^{r-1}$ for any $r\ge1$, then $\Sigma_k(A)=G$ for $3\le k\le m-3$.
\item For every $G$, one has $\Sigma_k(A)=G$ for $4\le k\le m-4$.
\end{enumerate}
\end{theorem}
\begin{proof}
For part~\textup{(a)}, apply \cref{lem:symmetry}(ii) with $n=N/2+1$ and $u=3$, and then use Roth and Lempel \cite[Theorem~3.1 and Remark~3.1]{RothLempel}. For part~\textup{(b)}, if $N=12$, the case $m=7$ is empty, while for $m\ge8$ the same theorem gives $\Sigma_4(B)=G$ for every $8$-element subset $B$, so \cref{lem:symmetry}(ii) with $n=8$ and $u=4$ applies. If $N\ge14$, apply \cref{lem:symmetry}(ii) with $n=N/2+1$ and $u=4$, and again use Roth--Lempel.
\end{proof}

\section{Small restricted sums}\label{sec:small-sums}
The proofs in this section use one simple representation-counting device. Its basic idea comes from Lev's proof of the three-fold restricted addition theorem \cite{LevIII}: a missing restricted sum forces strong restrictions on the two-fold representation function, while a sum with sufficiently many representations can be chosen to avoid previously selected elements.

\paragraph{A representation-counting principle.}
For $A\subset G$, put $m=|A|$ and define, for $x\in G$,
\[
\nu(x)=|\{(a,b)\in A^2:a\ne b,\ a+b=x\}|.
\]
The involution $(a,b)\mapsto(b,a)$ has no fixed point on the ordered restricted representations of any $x$. Consequently $\nu(x)$ is even. For $j\ge1$ set
\[
P_j=P_j(A):=\{x\in G:\nu(x)>2j\}.
\]
Then
\[
\sum_{x\in G}\nu(x)=m(m-1),\qquad \nu(x)\le M_A:=2\floor{m/2}.
\]
Moreover,
\begin{equation}\label{eq:representation-inclusion}
P_j+\Sigma_j(A)\subset \Sigma_{j+2}(A).
\end{equation}
Indeed, fix $x\in P_j$ and a restricted $j$-term representation $y=a_1+\cdots+a_j$. A prescribed element $a_i$ occurs in at most two ordered representations of $x$. Hence the $j$ forbidden elements exclude at most $2j$ ordered representations, while $\nu(x)>2j$. One representation $x=u+v$ avoids all $a_i$, and then $x+y$ is a restricted $(j+2)$-term sum.

Whenever $M_A>2j$, splitting the representation count on $P_j$ and its complement gives
\begin{equation}\label{eq:representation-count}
m(m-1)\le2j(N-|P_j|)+M_A|P_j|,\qquad
|P_j|\ge \frac{m(m-1)-2jN}{M_A-2j}.
\end{equation}
We use the same principle for the three-, four-, and five-fold endpoints.

The $2/5$ threshold in part~\textup{(a)} was proved by Lev \cite[Main Theorem]{LevIII}. We include a short proof here, based on his argument.

\Needspace{12\baselineskip}
\begin{proposition}\label{prop:three-fold}
Let $G$ be a finite abelian group of odd order $N$, and let $A\subset G$.
\begin{enumerate}[label=\textup{(\alph*)},leftmargin=2.5em]
\item If $N\ge149$ and $|A|>2N/5$, then $\Sigma_3(A)=G$.
\item If $N\ge17$ and $|A|>N/2$, then $\Sigma_3(A)=G$.
\end{enumerate}
\end{proposition}
\begin{proof}
The cases $N=17,19$ in part~\textup{(b)} follow from \cref{thm:DSH}. We prove part~\textup{(a)} for $N\ge149$ and part~\textup{(b)} for $N\ge21$.

Suppose that $g\notin\Sigma_3(A)$, put $m=|A|$, and let $P=\{x\in G:\nu(x)>2\}$. We shall show that the missing sum forces $A-A$ to be too small. First we obtain a uniform lower bound for the intersections $A\cap(A-d)$; we then combine it with a two-sum having many representations to place a translate of $A-A$ inside $P$. In both ranges under consideration, $m(m-1)>2N$, so $P\ne\varnothing$. By \eqref{eq:representation-inclusion}, $P+A\subseteq\Sigma_3(A)$. Hence $P\cap(g-A)=\varnothing$, and therefore
\begin{equation}\label{eq:three-P-upper}
|P|\le N-m.
\end{equation}
For $a\in A$, put $I_a=(a-g+A)\cap(-A)$. If $a-g+s=-t$ with $s,t\in A$, then $a+s+t=g$; since this is not a restricted three-term representation, at least two of $a,s,t$ coincide. The cases $s=a$, $t=a$, and $s=t$ contribute at most one point each, the last because doubling is bijective in an odd-order group. Hence $|I_a|\le3$.

For $d=a-b\in A-A$, let $B_d=\{u\in A:u+d\in A\}$. The intersection of $a-g+A$ and $b-g+A$ has size $|B_d|$, while their union is contained in $(G\setminus(-A))\cup I_a\cup I_b$. Therefore
\begin{equation}\label{eq:three-difference-count}
|B_d|\ge3m-N-6.
\end{equation}
Since $\nu(x)\le2$ outside $P$, averaging over $P$ gives
\begin{equation}\label{eq:three-P-average}
\max_{x\in P}\nu(x)
\ge 2+\frac{m(m-1)-2N}{|P|}
\ge 2+\frac{m(m-1)-2N}{N-m}.
\end{equation}
The last quantity is larger than $N-2m+9$. This is equivalent to $F(N,m):=-N^2+3Nm-9N-m^2+6m>0$. The function $F(N,m)$ is increasing in $m$ for $m\le N$. For part~\textup{(a)}, a direct calculation at the smallest admissible value $m=\lfloor2N/5\rfloor+1$ gives $F(N,m)>0$ for $N\ge149$; for part~\textup{(b)}, the same calculation at $m=\lfloor N/2\rfloor+1$ gives $F(N,m)>0$ for $N\ge21$. Thus there is $\sigma\in P$ with
\begin{equation}\label{eq:three-sigma-choice}
\nu(\sigma)>N-2m+9.
\end{equation}

We claim that $\sigma+(A-A)\subset P$. Put $C_\sigma=\{u\in A:\sigma-u\in A,\ \sigma-u\ne u\}$, so $|C_\sigma|=\nu(\sigma)$. For $d\in A-A$ and $u\in B_d\cap C_\sigma$, the pair $(u+d,\sigma-u)$ represents $\sigma+d$. Distinct $u$ give distinct ordered pairs, and at most one is diagonal. Hence, by \eqref{eq:three-difference-count} and \eqref{eq:three-sigma-choice},
\[
\nu(\sigma+d)\ge |B_d\cap C_\sigma|-1
\ge |B_d|+\nu(\sigma)-m-1>2.
\]
Therefore $\sigma+(A-A)\subset P$, and so
\begin{equation}\label{eq:three-difference-upper}
|A-A|\le |P|\le N-m.
\end{equation}

For part~\textup{(b)}, the trivial bound $|A-A|\ge m>N-m$ contradicts \eqref{eq:three-difference-upper}. For part~\textup{(a)}, let $H=\Stab(A-A)$ and put $h=|H|$. Kneser's theorem gives $|A-A|\ge2m-h$. Since \eqref{eq:three-difference-upper} gives $|A-A|<3N/5$, we obtain $h>2m-3N/5>N/5$. The case $H=G$ is already impossible, so $H$ is proper. Hence $1<[G:H]<5$; since $N$ is odd, $[G:H]=3$. Now $m>2N/5=6h/5>h$, so $A$ meets at least two $H$-cosets and $|A+H|\ge2h$. A second application of Kneser's theorem yields $|A-A|\ge2|A+H|-h\ge3h=N$, a contradiction.
\end{proof}

\begin{proposition}\label{prop:four-half}
Let $G$ be a finite abelian group of odd order $N\ge25$. If $|A|>N/2$, then $\Sigma_4(A)=G$.
\end{proposition}
\begin{proof}
Write $N=2n+1$ and, after replacing $A$ by a subset, assume $m=|A|=n+1$. Put $M=2\floor{m/2}$ and $p_2=|P_2|$. Suppose $g\notin\Sigma_4(A)$. By \eqref{eq:representation-inclusion}, $P_2+\Sigma_2(A)\subseteq\Sigma_4(A)$, and therefore
\begin{equation}\label{eq:p2-complement}
p_2\le N-|\Sigma_2(A)|.
\end{equation}

We first note that $\Sigma_2(A)\ne G$. Otherwise \eqref{eq:representation-count} gives $P_2\ne\varnothing$, and then $P_2+\Sigma_2(A)\subseteq\Sigma_4(A)$ would imply $\Sigma_4(A)=G$. Since $m>N/2$, we have $A+A=G$, so $\Sigma_2(A)\ne A+A$. As $G$ has odd order, $|G[2]|=1$, and \cref{thm:lev-two} together with \eqref{eq:p2-complement} gives
\begin{equation}\label{eq:p2-Lev}
p_2<N-\frac{1+\sqrt5}{2}m+3.
\end{equation}
On the other hand, \eqref{eq:representation-count} with $j=2$ gives
\begin{equation}\label{eq:p2-lower}
p_2\ge\frac{m(m-1)-4N}{M-4}.
\end{equation}

Put $\varphi=(1+\sqrt5)/2$. If $n$ is even, then $M=n$ and \eqref{eq:p2-lower} gives
\[
p_2\ge n-3-\frac{16}{n-4}>(2-\varphi)n+4-\varphi=N-\varphi m+3;
\]
if $n$ is odd, then $M=n+1$ and
\[
p_2\ge n-4-\frac{16}{n-3}>(2-\varphi)n+4-\varphi=N-\varphi m+3.
\]
The first strict inequality holds for every even $n\ge12$, and the second for every odd $n\ge13$; in each case the difference is increasing in $n$ and is already positive at the first admissible value. Both contradict \eqref{eq:p2-Lev}. Hence $\Sigma_4(A)=G$.
\end{proof}

\begin{proposition}\label{prop:five-half}
Let $G$ be a finite abelian group of odd order $N\ge29$. If $|A|>N/2$, then $\Sigma_5(A)=G$.
\end{proposition}
\begin{proof}
Write $N=2n+1$ and, after replacing $A$ by a subset, assume $m=|A|=n+1$.

The count \eqref{eq:representation-count} gives
\begin{equation}\label{eq:five-P-bounds}
|P_1|\ge n-2,\qquad |P_3|\ge n-9.
\end{equation}
Indeed, the contrary bounds $|P_1|\le n-3$ and $|P_3|\le n-10$ would
respectively give
\[
m(m-1)\le2N+(m-2)(n-3),\qquad
m(m-1)\le6N+(m-6)(n-10),
\]
contradicting $N\ge29$.  In the sole borderline calculation $N=29$
for the second inequality, use the even bound $\nu\le14$, which would
give $210\le206$.

We claim that $P_1+P_3+A=G$.  Otherwise let $H$ be its stabilizer,
$q=[G:H]>1$, and let $r_1,r_3,r_A$ count the $H$-cosets meeting the
three sets.  The quotient sum is aperiodic and proper, so Kneser's theorem gives
\begin{equation}\label{eq:r-sum-five}
r_1+r_3+r_A\le q+1.
\end{equation}
For $q\ge10$, \eqref{eq:five-P-bounds} and $|A|=n+1$ give
\[
r_1+r_3+r_A\ge
\frac{(n-2)q}{N}+\frac{(n-9)q}{N}+\frac{(n+1)q}{N}
=\frac{(3n-10)q}{N}
\ge\frac{(3n-10)q}{2n+1}>q+1,
\]
where the last inequality follows from $n\ge14$ and $q\ge10$.
For $2\le q\le9$, since $(n-2)/N\ge12/29$,
$(n-9)/N\ge5/29$, and $|A|>N/2$,
\[
r_1+r_3+r_A\ge
\left\lceil\frac{12q}{29}\right\rceil+
\left\lceil\frac{5q}{29}\right\rceil+
\left\lfloor\frac q2\right\rfloor+1>q+1.
\]
Both bounds contradict \eqref{eq:r-sum-five}.  Thus $P_1+P_3+A=G$.

By \eqref{eq:representation-inclusion},
\[
G=P_1+P_3+A
\subseteq P_3+\Sigma_3(A)
\subseteq\Sigma_5(A)
\subseteq G.
\]
Hence $\Sigma_5(A)=G$.
\end{proof}

\section{The sharp one-third theorem in odd order}\label{sec:one-third}
We prove the four- and five-fold endpoints first, then extend them by
padding and quotient selection.  The small quotients $C_3,C_5$ require fibre
arguments; larger least primes admit a popular-sum proof.
For a quotient $\pi:G\to C_p$ with kernel $H$, write
\[
A_i=A\cap\pi^{-1}(i),\qquad i+H=\pi^{-1}(i)\quad(i\in C_p).
\]
We first record the restricted-doubling input of Hamidoune, Llad\'o and Serra.

\begin{theorem}[Hamidoune--Llad\'o--Serra]\label{thm:HLS-third}
Let $H$ be a finite abelian group of odd order, and let $B\subset H$ satisfy
$0\in B$, $\langle B\rangle=H$, and $|B|\ge21$.  Then
\[
|\Sigma_2(B)|\ge \min\left\{|H|,\frac{3|B|}{2}\right\}.
\]
\end{theorem}
\begin{proof}
This is the odd-order bound of Hamidoune--Llad\'o--Serra \cite[p.~513]{HLS}.
\end{proof}

The following lemma is the only mixed restricted-doubling statement needed below.

\Needspace{20\baselineskip}
\begin{lemma}\label{lem:mixed-doubling-third}
Let $H$ be a finite abelian group of odd order $h$, and let $B,C\subset H$
with $|B|,|C|\ge21$.
\begin{enumerate}[label=\textup{(\roman*)},leftmargin=2.5em]
\item If $|B|+|C|>2h/3+1$, then $\Sigma_2(B)+\Sigma_2(C)=H$.
\item If every prime divisor of $h$ is at least $5$ and
$|B|+|C|>2h/3$, then the same conclusion holds.
\end{enumerate}
\end{lemma}
\begin{proof}
Put $K_B=\langle B-B\rangle$ and $K_C=\langle C-C\rangle$.
After translating into these subgroups, \cref{thm:HLS-third} gives
\[
|\Sigma_2(B)|\ge\min\{|K_B|,3|B|/2\},\qquad
|\Sigma_2(C)|\ge\min\{|K_C|,3|C|/2\}.
\]
If both subgroup orders exceed the corresponding second terms, then
$|\Sigma_2(B)|+|\Sigma_2(C)|>h$, so their sum is $H$.
Otherwise, interchange $B,C$ if necessary so that $|K_B|\le3|B|/2$.
The first inequality then says that $\Sigma_2(B)$ is a full $K_B$-coset.
It remains to show that $\Sigma_2(C)$ meets every coset of $K_B$.

Put $d=|K_B|$ and $q=[H:K_B]$; the case $q=1$ is immediate.
Suppose a class $t\in H/K_B$ is absent from the quotient image of
$\Sigma_2(C)$.  The involution $x\mapsto t-x$ has one fixed point, since
the quotient has odd order.  At most one coset in each non-fixed pair
can meet $C$, and the fixed coset can contain at most one element of $C$.
Otherwise two distinct elements of $C$ would have quotient sum $t$.
Thus $|C|\le(q-1)d/2+1$ and $|B|\le d$.

In part~\textup{(i)}, $q\ge3$, and the hypothesis gives
$2h/3+1<|B|+|C|\le(q-1)d/2+1+d\le2h/3+1$, a contradiction.
In part~\textup{(ii)}, we have $q\ge5$ and $d\ge|B|\ge21$, so
$2h/3<|B|+|C|\le(q-1)d/2+1+d<2h/3$, again a contradiction; the last
inequality is equivalent to $(q-3)d/6>1$.  Therefore $\Sigma_2(C)$ meets every
$K_B$-coset, and adding the full coset $\Sigma_2(B)$ proves the lemma.
\end{proof}

\subsection{The four-fold endpoint}\label{subsec:four-endpoint}

For a finite subset $X$ of an abelian group, put
\[
r_X(z)=|\{(x,y)\in X^2:x+y=z\}|,
\qquad X+_iX=\{z:r_X(z)\ge i\}\quad(i\ge1).
\]
We use the following equal-set form of Grynkiewicz's theorem on popular sums.

\Needspace{16\baselineskip}
\begin{theorem}[Grynkiewicz's theorem on popular sums]\label{thm:popular-sum}
Let $t\ge1$ be an integer, and let $X$ be a finite subset of an abelian
group with $|X|\ge t$.  Then either
\[
\sum_{i=1}^{t}|X+_iX|\ge 2t|X|-2t^2+1,
\]
or $L=\Stab(X+_tX)$ is non-trivial and
\[
\sum_{i=1}^{t}|X+_iX|\ge 2t|X|-t|L|.
\]
\end{theorem}
\begin{proof}
This is the equal-set specialization of \cite[Theorem~1.1]{Grynkiewicz}.
\end{proof}

The next proof compares the sum of the representation levels with an upper
bound forced by a missing restricted four-sum.

\begin{proposition}\label{prop:four-large-prime}
Let $G$ be a finite abelian group of odd order $N>291$ with $p(G)\ge7$.
If $A\subset G$ and $|A|>N/3$, then $\Sigma_4(A)=G$.
\end{proposition}
\begin{proof}
Fix $g\in G$. Choose $c\in G$ with $4c=g$, and put $B=A-c$, $m=|B|$, and $r(x)=r_B(x)$.
It suffices to prove $0\in\Sigma_4(B)$, since
$\Sigma_4(B)=\Sigma_4(A)-g$.  Suppose otherwise.  We compare the upper
bound on the truncated representation count forced by this missing
zero-sum with the lower bounds in \cref{thm:popular-sum}.

First observe that
\begin{equation}\label{eq:popular-opposite-zero}
r(x)\ge6\quad\Longrightarrow\quad r(-x)\le1
\qquad(x\in G).
\end{equation}
Indeed, distinct unordered restricted pairs with a fixed sum are disjoint,
and each sum has at most one diagonal representation.  Thus $r(x)\ge6$
gives at least three disjoint restricted pairs of sum $x$.
Any restricted pair of sum $-x$ meets at most two of these three pairs.
A remaining disjoint pair would give a restricted four-term zero-sum.
At zero we have the stronger bound $r(0)\le3$: two different unordered
zero-sum pairs would already give four distinct elements summing to zero.

We now use the level $21$. From \eqref{eq:popular-opposite-zero}, $\min\{21,r(x)\}+\min\{21,r(-x)\}\le22$ for every $x\in G$.
Indeed, if one count is at least six, the other is at most one; otherwise
both are at most five.  At $x=0$ the left side is at most six.
Counting the representation levels and pairing $x$ with $-x$ gives
\begin{align}
\sum_{i=1}^{21}|B+_iB|
&=\sum_{x\in G}\min\{21,r(x)\}\notag\\
&=\frac12\sum_{x\in G}
  \bigl(\min\{21,r(x)\}+\min\{21,r(-x)\}\bigr)\notag\\
&\le\frac{22(N-1)+6}{2}=11N-8.\label{eq:popular-upper}
\end{align}
The first equality counts a sum once for each level it reaches.

Now apply \cref{thm:popular-sum} with $t=21$; here $m>N/3>21$.
In its first alternative,
\[
\sum_{i=1}^{21}|B+_iB|
\ge42m-881>14N-881>11N-8,
\]
where the last inequality is exactly $N>291$.  This contradicts
\eqref{eq:popular-upper}.

In its second alternative, put $L=\Stab(B+_{21}B)$.
The level set is non-empty because
\[
\max_x r(x)\ge\frac{m^2}{N}>\frac N9>21,
\]
and it is proper because $r(0)\le3$.  Thus $L$ is a proper non-trivial
subgroup.  By $p(G)\ge7$, we have $|L|\le N/7$.
The second lower bound therefore gives
\[
\sum_{i=1}^{21}|B+_iB|
\ge42m-21|L|>14N-3N=11N,
\]
again contradicting \eqref{eq:popular-upper}.
Hence $0\in\Sigma_4(B)$, or equivalently $g\in\Sigma_4(A)$.
As $g$ was arbitrary, the proposition follows.
\end{proof}

\begin{proposition}\label{prop:sharp-four}
Let $G$ be a finite abelian group of odd order $N>605$, and let $A\subset G$.
If $|A|>N/3+1$, then $\Sigma_4(A)=G$.  Moreover, if $p(G)=5$, then the
same conclusion already follows from $|A|>N/3$.
\end{proposition}
\begin{proof}
For $p(G)\ge7$, the conclusion follows from \cref{prop:four-large-prime}.
We treat the quotients of orders three and five separately. In each case we either find a large fibre that already supplies the kernel part, or fix a target coset and use Kneser's theorem to show that the relevant fibre sumset fills it.

\proofcase{Case $p(G)=3$}
Choose $G/H\cong C_3$, and write $h=|H|>201$ and $a_i=|A_i|$ for $i\in C_3$.
The size hypothesis and integrality give
\begin{equation}\label{eq:balanced-p3-total}
a_0+a_1+a_2=|A|\ge h+2.
\end{equation}
In particular, at least two fibres are non-empty.

\proofcase{Large fibres}
Suppose only $A_i,A_j$ are occupied, with $a_i\ge a_j$.
Then $a_j\ge2$ and $a_i>(h+1)/2$.
The odd-order two-fold covering fact \cite[Lemma~2.2]{GallardoEtAl},
\cref{prop:three-fold}(b), and \cref{prop:four-half} give
\[
\Sigma_2(A_i)=2i+H,\qquad
\Sigma_3(A_i)=H,\qquad
\Sigma_4(A_i)=i+H.
\]
Using respectively $(4,0)$, $(3,1)$ and $(2,2)$ elements from these two fibres gives the three quotient sums $i$, $j$ and $2i+2j$. Thus
$\Sigma_4(A)=G$.

Now suppose all three fibres are non-empty.  If $a_i>h/2$, and $j,k$ are
the other two classes, then directly
\[
\Sigma_4(A_i)=i+H,\qquad
\Sigma_3(A_i)+A_j=j+H,\qquad
\Sigma_3(A_i)+A_k=k+H.
\]
Thus we may assume
\begin{equation}\label{eq:balanced-p3-upper}
a_0,a_1,a_2\le h/2.
\end{equation}

\proofcase{A class $i$ with $a_i\le h/3$}
Fix $i\in C_3$, and let $j,k$ be the other two classes.
If $a_i\le h/3$, then
\[
a_j+a_k\ge h+2-a_i>\frac{2h}{3}+1.
\]
Together with \eqref{eq:balanced-p3-upper}, this implies
$a_j,a_k>h/6+1>21$. By \cref{lem:mixed-doubling-third}(i), $\Sigma_2(A_j)+\Sigma_2(A_k)=i+H$, since $2j+2k=i$ in $C_3$.

\proofcase{A class $i$ with $h/3<a_i\le h/2$}
Put $X=\Sigma_2(A_i)$.  Every proper subgroup of $H$ has order at most
$h/3$, so translating an element of $A_i$ to zero produces a generating
subset of $H$.  Its size exceeds $21$, and \cref{thm:HLS-third} gives
\begin{equation}\label{eq:balanced-p3-X}
|X|\ge\frac32a_i.
\end{equation}
We claim that
\begin{equation}\label{eq:balanced-p3-target}
X+A_j+A_k=i+H.
\end{equation}
The quotient sum is correct because $2i+j+k=i$.
Suppose the claim fails, and let $L\le H$ be the stabilizer of the left
side.  This non-empty sum is $L$-periodic and omits at least one
$L$-coset.  Kneser's theorem therefore gives
\begin{equation}\label{eq:p3-kneser-short}
|X+L|+|A_j+L|+|A_k+L|\le h+|L|.
\end{equation}
By \eqref{eq:balanced-p3-total} and \eqref{eq:balanced-p3-X}, the left side is at least $\frac32a_i+a_j+a_k=|A|+a_i/2>7h/6+2$.
Hence $|L|>h/6+2$, so the odd index $[H:L]$ is either $3$ or $5$.
Write $d=|L|$.  We rule out each index by counting occupied cosets.

If $[H:L]=3$, then $|X|>3d/2$, so $X$ occupies at least two $L$-cosets.
The bound \eqref{eq:p3-kneser-short} permits at most four occupied cosets
in total.  Since $A_j,A_k$ are non-empty, $X$ must occupy exactly two
and each of $A_j,A_k$ exactly one.  Thus
\[
a_j,a_k\le d,\qquad a_i\ge(3d+2)-2d=d+2.
\]
If $A_i$ occupies exactly two $L$-cosets, then $a_i\ge d+2$ implies
that each contains at least two elements.  Hence the restricted quotient
sums $2u,u+v,2v$ are all realized and give the three classes of
$H/L\cong C_3$.  If $A_i$ occupies all three cosets, pairs from different
cosets again realize all three classes.  Either way, $X$ occupies at least
three $L$-cosets, giving at least five occupied $L$-cosets in total,
contrary to the four-coset bound from \eqref{eq:p3-kneser-short}.

If $[H:L]=5$, then $|X|>5d/2$, so $X$ occupies at least three cosets.
The six-coset bound \eqref{eq:p3-kneser-short} implies
\[
a_j+a_k\le3d,\qquad a_i\ge(5d+2)-3d=2d+2.
\]
Applying \eqref{eq:balanced-p3-X} once more gives $|X|\ge3d+3$.
Now $X$ occupies at least four cosets, while
\[
a_j+a_k=|A|-a_i\ge h+2-h/2>2d
\]
requires at least three in total for $A_j,A_k$.
Seven occupied cosets contradict the six-coset bound from
\eqref{eq:p3-kneser-short}.  This proves
\eqref{eq:balanced-p3-target}.  Since $i+H$ was arbitrary,
$\Sigma_4(A)=G$.

\proofcase{Case $p(G)=5$}
We prove the stronger assertion under $a:=|A|>N/3$.
Choose $G/H\cong C_5$, and put $h=|H|\ge125$ and $a_i=|A_i|$ for $i\in C_5$.
Indeed, $5h>605$ and every prime divisor of $h$ is at least five.

\proofcase{A fibre larger than $h/2$}
Translate so that $a_0>h/2$.
By \cref{prop:three-fold}(b) and \cref{prop:four-half},
\[
\Sigma_3(A_0)=\Sigma_4(A_0)=H.
\]
Thus $H$ is covered, and so is every occupied non-zero fibre class, since
$\Sigma_3(A_0)+A_x=x+H$.
If a second fibre satisfies $a_x>h/2$, then
\[
\Sigma_2(A_0)+\Sigma_2(A_x)=2x+H,\qquad
\Sigma_3(A_x)+A_0=3x+H,\qquad
\Sigma_4(A_x)=4x+H.
\]
The first equality uses \cref{lem:mixed-doubling-third}(ii), and together
the three equalities give all five cosets.

We may therefore assume $a_x\le h/2$ for every $x\ne0$.
The number of elements outside $A_0$ satisfies
$a-a_0>5h/3-a_0\ge2h/3$, so at least two non-zero fibres are occupied.
If all four are occupied, there is nothing left to prove.

If exactly three non-zero fibres are occupied, let $s$ be the missing
class, put $u=s/2$, and call the other occupied classes $x,y$.
In $C_5$ we have $2u=x+y=s$, and
\[
a_0+a_u=a-a_x-a_y>\frac{5h}{3}-h=\frac{2h}{3}.
\]
For $a_u\ge21$, by \cref{lem:mixed-doubling-third}(ii) we have
$\Sigma_2(A_0)+\Sigma_2(A_u)=s+H$.
For $a_u\le20$, we instead have
$a_0>2h/3-20>(h+1)/2$, so $\Sigma_2(A_0)=H$ and
$\Sigma_2(A_0)+A_x+A_y=s+H$.

If exactly two non-zero fibres $A_x,A_y$ are occupied, then
$a_x,a_y\le h/2$ gives $a_x+a_y\le h$, and hence $a_0>2h/3$.
Also $a_x+a_y=a-a_0>2h/3$, so each of $a_x,a_y$ is greater than $h/6$ and in particular at least two. Since $a_0>2h/3$, the odd-order two-fold covering result \cite[Lemma~2.2]{GallardoEtAl}, \cref{prop:three-fold}(b), and \cref{prop:four-half} give
\[
\Sigma_4(A_0)=H,\qquad
\Sigma_3(A_0)+A_x=x+H,\qquad
\Sigma_3(A_0)+A_y=y+H,
\]
\[
\Sigma_2(A_0)+\Sigma_2(A_x)=2x+H,\qquad
\Sigma_2(A_0)+\Sigma_2(A_y)=2y+H,\qquad
\Sigma_2(A_0)+A_x+A_y=x+y+H.
\]
Since $\{0,x,y,2x,2y,x+y\}=C_5$, all five cosets are covered.
This completes the large-fibre case.

\proofcase{All fibres have size at most $h/2$}
Assume $a_i\le h/2$ for every $i$.  Fix $g\in G$.  We show that
$g+H\subseteq\Sigma_4(A)$.  Choose $c\in G$ with $4c=g$, and put
$B=A-c$.  Since $\Sigma_4(B)=\Sigma_4(A)-g$, it is enough to prove
$H\subseteq\Sigma_4(B)$.  Translation only permutes the fibres, so we
relabel the fibres of $B$ as $A_i$, write $a_i=|A_i|$, and put $b=a_0$.

If $b\le h/3$, the four non-zero fibres have total size greater than
$4h/3$.  Hence, for some $x\in C_5\setminus\{0\}$,
$a_x+a_{-x}>2h/3$.  Both sizes exceed $h/6>20$, so \cref{lem:mixed-doubling-third}(ii) gives $\Sigma_2(A_x)+\Sigma_2(A_{-x})=H$.

It remains to consider $h/3<b\le h/2$.  Put $X=\Sigma_2(A_0)$ and
choose $x\in C_5\setminus\{0\}$ such that $a_x+a_{-x}$ is maximal.
Then
\begin{equation}\label{eq:balanced-p5-pair}
a_x+a_{-x}\ge\frac{a-b}{2}>\frac{7h}{12}.
\end{equation}
Both fibres are non-empty because each has size at most $h/2$.
A translate of $A_0$ generates $H$, and \cref{thm:HLS-third} gives
\begin{equation}\label{eq:balanced-p5-X}
|X|\ge\frac32b.
\end{equation}
Consequently
\[
|X|+a_x+a_{-x}\ge\frac32b+\frac{a-b}{2}
>b+\frac{5h}{6}>\frac{7h}{6}.
\]
Suppose $X+A_x+A_{-x}\ne H$ and let $L$ be its stabilizer.
Exactly as in \eqref{eq:p3-kneser-short}, Kneser's theorem gives
\begin{equation}\label{eq:p5-kneser-short}
|X+L|+|A_x+L|+|A_{-x}+L|\le h+|L|.
\end{equation}
Thus $|L|>h/6$.  Since all prime divisors of $h$ are at least five,
necessarily $[H:L]=5$.

We use this six-coset bound twice.  First, $|X|>h/2$ requires at least
three $L$-cosets, leaving at most three for $A_x,A_{-x}$ in total.
Hence $a_x+a_{-x}\le3h/5$.  By \eqref{eq:balanced-p5-pair},
$a-b\le2(a_x+a_{-x})\le6h/5$, and therefore
$b>5h/3-6h/5=7h/15$.  By \eqref{eq:balanced-p5-X},
$|X|>(3/2)(7h/15)=7h/10$.  Thus $X$ occupies at least four
$L$-cosets.  But \eqref{eq:balanced-p5-pair} requires at least three
in total for $A_x,A_{-x}$, because $7h/12>2h/5$.
Seven occupied cosets contradict the six-coset bound from
\eqref{eq:p5-kneser-short}.  Hence $X+A_x+A_{-x}=H$.

Thus $H\subseteq\Sigma_4(B)$, and consequently
$g+H\subseteq\Sigma_4(A)$.  Since $g$ was arbitrary,
$\Sigma_4(A)=G$.
\end{proof}

\subsection{Five-fold covering and the remaining lengths}\label{subsec:long-lengths}
We use two standard results in quotient form.  Diderrich's theorem
\cite[Theorem~1]{Diderrich}, in the form recorded in
\cite[Remark~2]{GGH}, states that if $H$ has odd order,
$B\subset H$, $K\le H$, $\pi:H\to H/K$, and $1\le r<|B|$, then
\begin{equation}\label{eq:Diderrich}
|\pi(\Sigma_r(B))|\ge|\pi(B)|.
\end{equation}
We also use Girard--Griffiths--Hamidoune \cite[Theorem~5.1]{GGH}: for a sequence
$S$ with $|S|>|\supp S|$ and $2\le r\le|S|-2$,
\begin{equation}\label{eq:GGH}
|\Sigma_r(S)|\ge|\supp S|,
\end{equation}
with equality if and only if \textup{(i)} $\supp S$ is a coset of a subgroup,
\textup{(ii)} $|\supp S|=2$ and exactly one support element occurs at least
twice, or \textup{(iii)} $\supp S$ is a three-term arithmetic progression and
only its middle term occurs at least twice.

\begin{lemma}\label{lem:three-fibre}
Let $H$ have odd order $h$, and let $B_0,B_1,B_2\subset H$ be non-empty with $|B_0|+|B_1|+|B_2|\ge h+2$. Let $2\le r_i\le|B_i|-2$ for $i=1,2$, allowing also $(|B_i|,r_i)=(3,2)$, and put $S=B_0+\Sigma_{r_1}(B_1)+\Sigma_{r_2}(B_2)$. If $S\ne H$, let $K=\Stab(S)$, $q=[H:K]$, and let $b_i$ be the number of $K$-cosets meeting $B_i$. Then $b_1=b_2=1$ and $b_0=q-1$. Moreover, every occupied $K$-coset of each $B_i$ contains at least two elements.
\end{lemma}
\begin{proof}
Put $d=|K|$ and let $\pi:H\to H/K$.
The proof identifies the equality case in Kneser's theorem and then uses
the available mass to rule out the exceptional restricted-sum configurations.

For each restricted block, \eqref{eq:Diderrich} gives an image of size at
least $b_i$; the image of $B_0$ has size $b_0$.
Since $\pi(S)$ is aperiodic and proper, Kneser's theorem gives
\begin{equation}\label{eq:three-fibre-chain}
q-1\ge|\pi(S)|\ge b_0+b_1+b_2-2.
\end{equation}
On the other hand, the occupied cosets have enough capacity for the three sets:
\begin{equation}\label{eq:three-fibre-capacity}
(b_0+b_1+b_2)d\ge |B_0|+|B_1|+|B_2|\ge qd+2.
\end{equation}
The case $d=1$ is impossible.  As $b_0+b_1+b_2$ is an integer, the two
bounds force
\begin{equation}\label{eq:three-fibre-equality}
b_0+b_1+b_2=q+1.
\end{equation}
Every inequality in \eqref{eq:three-fibre-chain} is now an equality.
In particular, $|\pi(\Sigma_{r_i}(B_i))|=b_i$ for $i=1,2$.

The total capacity is $(q+1)d$, whereas the three sets have at least
$qd+2$ elements.  Their total deficiency from full occupied cosets is
therefore at most $d-2$.  A singleton occupied coset would contribute
deficiency $d-1$ on its own, which is impossible.
Thus every occupied coset of each $B_i$ contains at least two elements.

Fix $i\in\{1,2\}$.  If $(|B_i|,r_i)=(3,2)$, the deficiency argument already implies
$b_i=1$.  Otherwise project $B_i$ to $H/K$, retaining multiplicities,
and denote this sequence by $S_i$.
Every support element occurs at least twice, and
$2\le r_i\le|S_i|-2$.  Equality in \eqref{eq:GGH}, as classified in
\cite[Theorem~5.1]{GGH}, therefore forces $\supp S_i$ to be a subgroup coset:
the other two equality types each require a support point occurring only once.

The quotient restricted sumset is contained in the corresponding coset
of that subgroup and has the same cardinality, so it is the entire coset.
A non-trivial such subgroup would then be a period of $\pi(S)$,
contrary to its aperiodicity.  Hence $b_i=1$ for $i=1,2$, and
\eqref{eq:three-fibre-equality} gives $b_0=q-1$.
\end{proof}

\begin{proposition}\label{prop:sharp-five}
Let $G$ be a finite abelian group of odd order $N>605$, and let $A\subset G$.
If $|A|>N/3+1$, then $\Sigma_5(A)=G$.
\end{proposition}
\begin{proof}
For $p(G)\ge5$, delete any one element of $A$ and apply
\cref{prop:four-large-prime} or the stronger $p=5$ assertion in
\cref{prop:sharp-four}.  The remaining set has size greater than $N/3$;
adding the deleted element gives a restricted five-sum for every element of $G$.

Now let $p(G)=3$, choose $G/H\cong C_3$, and put
$h=|H|>201$, $a=|A|\ge h+2$, and $a_i=|A_i|$.
Translate so that $A_0$ is largest.

If $a_0>h/2$, choose $T\subset A_0$ with $|T|=(h+1)/2$.
By \cref{prop:three-fold}(b) and \cref{prop:four-half,prop:five-half},
\begin{equation}\label{eq:absorber-345-605}
\Sigma_3(T)=\Sigma_4(T)=\Sigma_5(T)=H.
\end{equation}
Since $a\ge h+2$, there are distinct $y_1,y_2\in A\setminus A_0$.
Their non-zero quotient residues satisfy
\[
\{0,\pi(y_1),\pi(y_2),\pi(y_1)+\pi(y_2)\}=C_3.
\]
For any quotient class, select zero, one, or two of these elements and use
respectively five, four, or three elements of $T$ to complete the sum in $H$.
The two parts are disjoint, so this proves the claim in this case.

Assume every $a_i\le(h-1)/2$.  Then each $a_i\ge3$, since the other two
fibres have total size at most $h-1$.
Fix a class $s\in C_3$, put $r=-s$, and let $j,t$ be the two classes distinct from $r$. Since $r+2j+2t=s$, the sum $A_r+\Sigma_2(A_j)+\Sigma_2(A_t)$ lies in $s+H$. If it equals $s+H$, we are done. Suppose otherwise, and let $K$ be its stabilizer, $d=|K|$, and $q=[H:K]$.
By \cref{lem:three-fibre}, each of $A_j,A_t$ occupies one $K$-coset, and $A_r$ occupies $q-1$. Hence $qd+2\le a\le(qd-1)/2+2d$, which forces $q=3$.  Moreover, both occupied $K$-cosets of $A_r$ contain
at least two elements.  Their restricted pairs realize all three classes
of $H/K$, and therefore
\begin{equation}\label{eq:Ar-double-605}
\Sigma_2(A_r)+K=2r+H.
\end{equation}

The mass in the other two fibres satisfies $a_j+a_t\ge(3d+5)/2$.
Without loss of generality, $a_j\ge a_t$. Then $a_j>(d+3)/2$ and
$d=h/3>67$. By \cref{prop:three-fold}(b), applied to a translate of
$A_j$ inside $K$, $\Sigma_3(A_j)$ is a full $K$-coset. Combining it with
\eqref{eq:Ar-double-605} gives
\[
\Sigma_2(A_r)+\Sigma_3(A_j)=2r+H=s+H.
\]
Thus $s+H$ is covered in either case. Since $s$ was arbitrary, $\Sigma_5(A)=G$.
\end{proof}

The four- and five-fold cases provide the base cases, from which we extend the odd-order covering result to all longer lengths.

\begin{proposition}\label{prop:one-third-long}
Let $G$ be a finite abelian group of odd order $N>605$, and let $A\subset G$.
If $|A|>N/3+1$, then
\[
\Sigma_k(A)=G\qquad(6\le k\le |A|-6).
\]
\end{proposition}
\begin{proof}
Put $a=|A|$.  By \cref{lem:symmetry}, it is enough to prove covering for
$6\le k\le a/2$.
We first use equal-sum pairs to cover a fixed interval of short lengths.
For the remaining lengths, we choose multiplicities with the required quotient sum
and then fill the corresponding coset with actual elements of the fibres.

Some sum $\beta$ has at least
\[
R\ge\left\lceil\frac{\binom a2}{N}\right\rceil
\ge\left\lceil\frac N{18}+\frac16\right\rceil\ge34
\]
unordered restricted representations.
Pairs with a fixed sum are disjoint, so these representations form a matching.
For $k=4+2\ell\le64$, use \cref{prop:sharp-four} to represent
$g-\ell\beta$ by four elements of $A$.
At most four matching pairs meet those elements, leaving at least
$R-4\ge30\ge\ell$ unused pairs.  Adding $\ell$ of them gives a restricted
$k$-sum equal to $g$.
For $k=5+2\ell\le63$, the same argument using \cref{prop:sharp-five}
leaves at least $R-5\ge29\ge\ell$ pairs.
Thus only $65\le k\le a/2$ remains.

If $N$ is prime, \cref{thm:DSH} already gives covering: throughout the
lower half,
\[
k(a-k)+1\ge4(a-4)+1>N.
\]
We may henceforth assume $N$ is composite.
Choose a quotient $\pi:G\to C_p$, where $p=p(G)$, and write $H=\ker\pi$, $h=|H|=N/p$, and $a_i=|A_i|$.
In particular $h\ge p$.

\proofcase{Case $p=3$}
Translate so that $A_0$ is a largest fibre.
If $a_0>h/2$, choose $T\subset A_0$ of size $(h+1)/2$ and distinct
$y_1,y_2\in A\setminus A_0$; the latter exist because $a\ge h+2$.
By \cref{prop:three-fold}(b) and \cref{prop:four-half,prop:five-half},
\[
\Sigma_3(T)=\Sigma_4(T)=\Sigma_5(T)=H.
\]
Choose a filler $F\subset A\setminus(T\cup\{y_1,y_2\})$ with
$|F|=k-5$.  The available number is sufficient since
\[
a-|T|-2-(k-5)\ge\frac a2-\frac{h+1}{2}+3>0.
\]
The subset sums of the two non-zero quotient residues of $y_1,y_2$ exhaust $C_3$. Select zero, one, or two of these elements to correct the quotient of $g-\sigma_F$, and complete using respectively five, four, or three elements of $T$. All parts are disjoint.

Now assume every fibre has size at most $h/2$.
Let $A_r$ be a smallest fibre and order the other two so that
$a_u\ge a_v\ge a_r$.
Because $h$ is odd, the other two fibres have total size at most $h-1$,
so $a_r\ge3$.  Also
\[
h+2\le a_u+a_v+a_r\le\frac{h-1}{2}+2a_v,
\]
and hence
\begin{equation}\label{eq:long-p3-sizes}
a_r\ge3,\qquad a_v\ge\frac{h+5}{4}.
\end{equation}

We want one summand from $A_r$, at least two from each of $A_u,A_v$,
and at least two elements left unused in each of the latter fibres.
After allowing for these two baseline and two spare elements, the
remaining capacities are $a_u-4,a_v-4$. Put $\ell=k-5$. To apply
\cref{lem:quotient-selection}, it is enough to check
\begin{equation}\label{eq:p3-selection-capacity}
\min\{a_u-4,\ell\}+\min\{a_v-4,\ell\}\ge\ell+2.
\end{equation}
If $a_u-4\ge\ell$, this follows at once from
\eqref{eq:long-p3-sizes}, since $a_v-4\ge2$. Otherwise both terms are
uncapped, and
\[
(a_u-4)+(a_v-4)=a-a_r-8\ge\frac{2a}{3}-8\ge\ell+2.
\]
Thus \cref{lem:quotient-selection} applies.

For a class $s\in C_3$, select an auxiliary $\ell$-sum with
quotient $s-r-2u-2v$.
Adding the baseline multiplicities gives integers $n_u,n_v$ with
\begin{gather*}
2\le n_u\le a_u-2,\qquad 2\le n_v\le a_v-2,\\
n_u+n_v=k-1,\qquad r+n_u u+n_v v=s.
\end{gather*}
We retain only these multiplicities, not the particular auxiliary
representing elements.  The full original fibres remain available for
the kernel-covering argument.

Consider
\[
A_r+\Sigma_{n_u}(A_u)+\Sigma_{n_v}(A_v)\subset s+H.
\]
If this sum were a proper subset of $s+H$, \cref{lem:three-fibre}
would put $A_u,A_v$ each in a single coset of its proper stabilizer
$K\le H$.
Since $A_r$ is smallest, we would then have
\[
a_r\le a_v\le a_u\le |K|,
\qquad a\le3|K|\le h,
\]
contrary to $a\ge h+2$.
Thus every class of $G/H$ is filled by restricted $k$-sums.

\proofcase{Case $p=5$}
Here $h\ge125$. Let $A_i,A_j$ be the two largest fibres, with $a_i\ge a_j$, and put $s=\lceil2h/3\rceil$.
Since $a_i+a_j\ge2a/5>2h/3$, we have $a_i+a_j\ge s$.
Moreover $a_j\ge21$: otherwise
$a\le h+4\cdot20<5h/3+1$, a contradiction. Choose
$T_i\subset A_i$ and $T_j\subset A_j$ with
$|T_i|,|T_j|\ge21$ and $|T_i|+|T_j|=s$. By
\cref{lem:mixed-doubling-third}(ii),
\begin{equation}\label{eq:p5-two-fibre-absorber}
\Sigma_2(T_i)+\Sigma_2(T_j)=(2i+2j)+H.
\end{equation}

Put $M=A\setminus(T_i\cup T_j)$ and $\ell=k-4\ge61$, and let $r_x$
be the fibre sizes of $M$. We claim
\begin{equation}\label{eq:p5-selection-capacity}
\sum_{x\in C_5}\min\{r_x,\ell\}\ge\ell+4.
\end{equation}
Indeed, $s<a/2$, so $|M|=a-s>a/2\ge k=\ell+4$. If every
$r_x<\ell$, this proves \eqref{eq:p5-selection-capacity}. If at least two
fibres have size at least $\ell$, then the left side is at least
$2\ell\ge\ell+4$. It remains to suppose that exactly one fibre, say the
$x$-fibre, has size at least $\ell$. If $x\in\{i,j\}$, then
$r_x\le h-21$ while $|M|>h$, and hence $|M|-r_x>21$. If
$x\notin\{i,j\}$ and $|M|-r_x\le3$, then the choice of $i,j$ gives
$a_i,a_j\ge r_x$, so $a\ge3r_x$. On the other hand,
$a=s+|M|\le s+r_x+3$, whence $r_x\le(s+3)/2$ and therefore
$a\le3(s+3)/2\le h+6$, contradicting $a>5h/3+1$. Thus in all cases
$|M|-r_x\ge4$, proving \eqref{eq:p5-selection-capacity}.

By \cref{lem:quotient-selection}, $\pi(\Sigma_\ell(M))=C_5$. Combining
this with \eqref{eq:p5-two-fibre-absorber}, we obtain
$G=\Sigma_2(T_i)+\Sigma_2(T_j)+\Sigma_\ell(M)\subseteq\Sigma_k(A)$.

\proofcase{Case $p\ge7$}
Let $A_i,A_j$ be the two largest fibres, with $a_i\ge a_j$, and set $s=\lceil2h/3\rceil$. Their combined size satisfies $a_i+a_j\ge2a/p>2h/3$, so $a_i+a_j\ge s$. Choose subsets $T_i\subset A_i$ and $T_j\subset A_j$ of total size $s$ such that
\begin{equation}\label{eq:large-p-two-fibre-absorber}
\Sigma_2(T_i)+\Sigma_2(T_j)=(2i+2j)+H.
\end{equation}

Suppose first that $h$ is prime. Since $ph>605$ and $p\le h$, we have
$h\ge29$. Also $a_j\ge2$, for otherwise
$a\le h+p-1\le ph/3+1$. Choose $|T_i|,|T_j|\ge2$ with total size
$s$. By \cref{thm:DSH}, either one of the two restricted pair sumsets
already fills its $H$-coset, or
$|\Sigma_2(T_i)|+|\Sigma_2(T_j)|\ge2s-6>h$. In either case
\eqref{eq:large-p-two-fibre-absorber} follows.

Suppose now that $h$ is composite. We have $a_j\ge21$. Indeed, if
$a_j\le20$, then $a\le h+20(p-1)$. For $p=7$, compositeness and
$7h>605$ give $h\ge91$; for $p\ge11$, all prime divisors of $h$ are at
least $p$, so $h\ge p^2$. In either case
$h+20(p-1)\le ph/3+1$, a contradiction. We may therefore choose
$|T_i|,|T_j|\ge21$ with total size $s$, and
\cref{lem:mixed-doubling-third}(ii) gives
\eqref{eq:large-p-two-fibre-absorber}.

Put $M=A\setminus(T_i\cup T_j)$ and $\ell=k-4\ge61$, and let $r_x$ be
the fibre sizes of $M$. Since $3\nmid h$, we have $s<2h/3+1$, and hence
$|M|=a-s>(p-2)h/3$. If $\ell\le h$, then
\[
\sum_x\min\{r_x,\ell\}\ge\frac{\ell}{h}|M|
>\frac{p-2}{3}\ell\ge p+\ell-1,
\]
where the last comparison is $\ell(p-5)\ge3(p-1)$.

For $\ell>h$, we have
$\sum_x\min\{r_x,\ell\}=|M|=a-s$. Since $a\ge2k$, this is at least
$k=\ell+4$. More precisely,
\[
|M|-\ell=a-s-k+4\ge\frac a2-s+4
>\frac{p-4}{6}h+\frac72>p-2.
\]
Thus integrality gives $|M|\ge p+\ell-1$. In either case
\cref{lem:quotient-selection} gives $\pi(\Sigma_\ell(M))=C_p$.
Together with \eqref{eq:large-p-two-fibre-absorber}, this yields
$G=\Sigma_2(T_i)+\Sigma_2(T_j)+\Sigma_\ell(M)\subseteq\Sigma_k(A)$.

All lower-half lengths have now been covered, and \cref{lem:symmetry}
completes the proof.
\end{proof}

\begin{proof}[Proof of \cref{thm:sharp-density}]
The four-fold endpoint follows from \cref{prop:sharp-four}, the five-fold
endpoint from \cref{prop:sharp-five}, and every length
$6\le k\le|A|-6$ from \cref{prop:one-third-long}.
By \cref{lem:symmetry}, the complementary endpoints $|A|-5$ and $|A|-4$
also cover $G$.
Therefore $\Sigma_k(A)=G$ for $4\le k\le|A|-4$.
The sharpness of the additive constant is given by the index-three instance
of the construction preceding \cref{thm:least-prime}.
\end{proof}

\begin{proof}[Proof of \cref{cor:critical-densities}]
The odd-order two-fold result \cite[Lemma~2.2]{GallardoEtAl} gives $c_2\le1/2$, and intervals in odd cyclic groups give the matching asymptotic lower bound. \Cref{thm:lev-three} gives $c_3\le2/5$ in odd order, while the union of two cosets of an index-five subgroup gives equality. For every $k\ge4$, \cref{thm:sharp-density} gives $c_k\le1/3$, while an index-three subgroup gives $c_k\ge1/3$.
\end{proof}

\section{The exact least-prime threshold}\label{sec:least-prime}
At length $2p-2$, the removal lemma leads to two cases. Either almost all of $A$ lies in one coset of an index-$p$ subgroup, or every element of $G$ has many representations as a sum of $2p-2$ elements of $A$. In the second case, we find a subset $T\subseteq A$ with $|T|\le |A|/2$ such that $\Sigma_{2p-2}(T)=G$. The remaining elements can then be added to obtain all longer lengths in the required range.

\begin{lemma}\label{lem:least-prime-rigidity}
Let $p$ be an odd prime, and let $G$ be a finite abelian group whose prime
divisors are all at least $p$.  Suppose $X_1,\ldots,X_{2p-2}\subseteq G$ satisfy
\[
|X_i|\ge\left(\frac1p-\frac{1}{10p^2}\right)|G|.
\]
If $X_1+\cdots+X_{2p-2}\ne G$, its stabilizer has index $p$, and at least
one $X_i$ lies in a single coset of the stabilizer.
\end{lemma}
\begin{proof}
Let $H$ be the stabilizer of $X_1+\cdots+X_{2p-2}$, put $q=[G:H]>1$,
and let $c_i$ count the $H$-cosets meeting $X_i$.
The quotient sum is aperiodic and proper.  Kneser's theorem therefore gives
\begin{equation}\label{eq:rigidity-budget}
\sum_{i=1}^{2p-2}(c_i-1)\le q-2,
\qquad
c_i\ge\left(\frac1p-\frac1{10p^2}\right)q.
\end{equation}
The second bound simply divides the size of $X_i$ by the coset size
$|G|/q$.  Since all prime divisors of $|G|$ are at least $p$, the index
$q\ge p$ is odd.

If $p<q<2p$, then
$(1/p-1/(10p^2))q>1$, so all $c_i\ge2$.
The left side of \eqref{eq:rigidity-budget} is then at least
$2p-2>q-2$, a contradiction.
For $q\ge2p+1$, the coefficient of $q$ in the lower bound obtained from
\eqref{eq:rigidity-budget} is positive, so its minimum occurs at $2p+1$:
\[
\sum_{i=1}^{2p-2}(c_i-1)-(q-2)\ge\frac{3p^2-9p+1}{5p^2}>0.
\]
This again contradicts \eqref{eq:rigidity-budget}.
The value $q=2p$ is impossible because $q$ is odd.  Thus $q=p$.

Now \eqref{eq:rigidity-budget} says $\sum_i(c_i-1)\le p-2$.
If every $c_i$ were at least two, this sum would be at least $2p-2$.
Hence some $c_i=1$, as required.
\end{proof}

The next lemma is the core combinatorial input in the proof of
\cref{thm:least-prime}.

\begin{lemma}\label{lem:counting-cover}
Fix an integer $s\ge3$ and $0<\delta\le1$.  Let $G$ be a finite abelian group
of odd order $N$, and let $A\subseteq G$.  Suppose every $g\in G$ has at least
$\delta N^{s-1}$ ordinary ordered representations as a sum of $s$ elements of $A$.
If
\begin{equation}\label{eq:counting-size-conditions}
N\ge\frac{2\binom{s}{2}}{\delta},
\qquad
N\exp\left(-\frac{\delta N}{2^{s+1}s^2}\right)<\frac12,
\end{equation}
then some $T\subseteq A$ satisfies $|T|\le |A|/2$ and
$\Sigma_s(T)=G$.
\end{lemma}
\begin{proof}
Put $m=|A|$.  Fix $g\in G$.  We first discard representations with repeated entries.
For two positions $i<j$, an ordinary representation with $a_i=a_j$ has at most
$N^{s-2}$ choices: after choosing the other $s-2$ entries, the repeated value is
uniquely determined by $2a_i=g-\sum_{t\ne i,j}a_t$, since doubling is a
bijection in the odd-order group $G$.  Hence at most
$\binom{s}{2}N^{s-2}$ representations have a collision.  By the first condition
in \eqref{eq:counting-size-conditions}, at least
\begin{equation}\label{eq:counting-restricted}
\delta N^{s-1}-\binom{s}{2}N^{s-2}
\ge\frac{\delta}{2}N^{s-1}
\end{equation}
ordered representations have all entries distinct.

From these restricted representations we greedily extract pairwise disjoint
representing sets.  A fixed element $a\in A$ occurs in at most $sN^{s-2}$
ordered representations of $g$: choose its position and then $s-2$ of the
remaining entries, after which the final entry is forced.  Thus at most
$s^2N^{s-2}$ ordered representations meet any prescribed $s$-element set.
Starting with the representations counted in \eqref{eq:counting-restricted},
retain the underlying set of one representation, delete all representations
meeting it, and repeat.  This produces pairwise disjoint sets
$E_{g,1},\ldots,E_{g,r_g}\subseteq A$, each of size $s$ and with sum $g$, where
\begin{equation}\label{eq:counting-disjoint-reps}
r_g\ge\frac{(\delta/2)N^{s-1}}{s^2N^{s-2}}=\frac{\delta N}{2s^2}.
\end{equation}
This construction is carried out separately for each $g$; only the sets within
one such family need to be disjoint.

For a fixed $g$, let $\mathcal B_g=\{T\subseteq A:g\notin\Sigma_s(T)\}$.
A set $T\in\mathcal B_g$ cannot contain any complete block $E_{g,j}$.  There
are $2^s-1$ possible intersections with each block, independently across the
disjoint blocks, while the remaining $m-sr_g$ elements are unrestricted.  Hence
\[
|\mathcal B_g|\le(2^s-1)^{r_g}2^{m-sr_g}=2^m(1-2^{-s})^{r_g}
\le2^m\exp(-r_g/2^s)\le2^m\exp\left(-\frac{\delta N}{2^{s+1}s^2}\right),
\]
where we used $1-u\le e^{-u}$ and \eqref{eq:counting-disjoint-reps}.  Summing over $g\in G$ gives
\begin{equation}\label{eq:counting-all-bad}
\left|\bigcup_{g\in G}\mathcal B_g\right|\le\sum_{g\in G}|\mathcal B_g|\le2^mN\exp\left(-\frac{\delta N}{2^{s+1}s^2}\right)<2^{m-1}.
\end{equation}
The last inequality is the second condition in
\eqref{eq:counting-size-conditions}.

Finally, pair each subset $T\subseteq A$ with its complement $A\setminus T$.
There are $2^{m-1}$ pairs, and each pair has a member of size at most
$\lfloor m/2\rfloor$. Therefore
$\bigl|\{T\subseteq A:|T|\le\lfloor m/2\rfloor\}\bigr|\ge2^{m-1}$.
By \eqref{eq:counting-all-bad}, not all of these subsets belong to
$\bigcup_{g\in G}\mathcal B_g$.  Any one outside the union satisfies
$\Sigma_s(T)=G$.
\end{proof}

\begin{proof}[Proof of \cref{thm:least-prime}]
For $p=3$, the result is \cref{thm:sharp-density}.
We therefore assume $p\ge5$.

Set $s=2p-2$ and
$\rho=\min\{1,\rho_s(1/(10p^2))\}>0$, where $\rho_s$ is supplied by
\cref{lem:removal}.
Choose $N_0(p)$ large enough that every $N\ge N_0(p)$ satisfies
\begin{equation}\label{eq:least-prime-order-conditions}
\frac Np>605,\qquad
N\ge\frac{2\binom{s}{2}}{\rho},\qquad
N\exp\left(-\frac{\rho N}{2^{s+1}s^2}\right)<\frac12.
\end{equation}
Let $G$ and $A$ satisfy the hypotheses, and write $h=N/p$ and $m=|A|\ge h+p-1$.
By \cref{lem:symmetry}, it suffices to prove $\Sigma_k(A)=G$ for
$s\le k\le m/2$.
\proofcase{Case 1. Every $g\in G$ has more than $\rho N^{s-1}$ ordinary ordered $s$-term representations from $A$}
By \cref{lem:counting-cover}, there is $T\subseteq A$ such that
$|T|\le m/2$ and $\Sigma_s(T)=G$. This one set supplies all longer lengths
in the lower half. Indeed, fix $s\le k\le m/2$ and choose
$F\subseteq A\setminus T$ with $|F|=k-s$. The choice is possible because
$|A\setminus T|\ge\lceil m/2\rceil$. Then
$G=\sigma_F+\Sigma_s(T)\subseteq\Sigma_k(A)$. Every representation on the
left uses $s$ elements of $T$ and the fixed $k-s$ elements of $F$. Their
disjointness ensures that the resulting sum is restricted. Symmetry
completes this case.

\proofcase{Case 2. Some $g_0\in G$ has at most $\rho N^{s-1}$ ordinary ordered $s$-term representations from $A$}
We first obtain concentration in an index-$p$ coset, then use the few
necessary outside elements to choose the quotient class. Apply
\cref{lem:removal} to $x_1+\cdots+x_s=g_0$, with $x_i\in A$, using deletion
parameter $1/(10p^2)$.
For each position it gives a subset $X_i\subseteq A$ with
\[
|X_i|\ge m-\frac{N}{10p^2}
\ge\left(\frac1p-\frac1{10p^2}\right)N,
\qquad
 g_0\notin X_1+\cdots+X_s.
\]
By \cref{lem:least-prime-rigidity}, the stabilizer $H$ of this proper
ordinary sumset has index $p$, and one of the $X_i$ lies in a single
$H$-coset.
Consequently that coset contains at least $m-N/(10p^2)$ elements of $A$.
Translate it to $H$ using \cref{lem:translation}, and put $B=A\cap H$.
Since $|H|=h$, we obtain
\begin{equation}\label{eq:least-prime-dense-fibre}
\begin{aligned}
|B|&\ge m-\frac{h}{10p}\ge h+p-1-\frac{h}{10p}\ge\frac{49h}{50},\\
m&\le |B|+\frac{h}{10p}\le h+\frac{h}{10p}\le\frac{51h}{50}.
\end{aligned}
\end{equation}
Moreover, $|A\setminus H|=m-|B|\ge m-h\ge p-1$. Thus $B$ contains almost all of $H$, and $A\setminus H$ still contains at least $p-1$ elements.

Choose distinct $y_1,\ldots,y_{p-1}\in A\setminus H$, and let $\pi:G\to G/H\cong C_p$.
Each $\pi(y_i)$ is non-zero, so repeated Cauchy--Davenport gives
\begin{equation}\label{eq:least-prime-quotient-choice}
\{0,\pi(y_1)\}+\cdots+\{0,\pi(y_{p-1})\}=C_p.
\end{equation}
Indeed, its lower bound is
$\min\{p,2(p-1)-(p-2)\}=p$.
Each two-element set in this ordinary sum corresponds to a different $y_i$. Therefore selecting its non-zero term uses each $y_i$ at most once, even when some quotient residues coincide.

Fix $s\le k\le m/2$ and $g\in G$. By \eqref{eq:least-prime-quotient-choice}, some subset of $\{y_1,\ldots,y_{p-1}\}$ has the same quotient sum as $g$. Suppose it has $j$ elements, where $0\le j\le p-1$; after relabelling, write this subset as $\{y_1,\ldots,y_j\}$. After subtracting their sum from $g$, the
remainder lies in $H$ and has to be represented using $k-j$ elements.
Its lower endpoint satisfies
$k-j\ge(2p-2)-(p-1)=p-1\ge4$. For the upper endpoint,
\eqref{eq:least-prime-dense-fibre} and $h>605$ give
$k-j\le m/2\le51h/100<|B|-4$.
Also $|B|\ge49h/50>h/3+1$.
By \cref{thm:sharp-density}, applied inside $H$,
$g-y_1-\cdots-y_j\in\Sigma_{k-j}(B)$. Hence
$g\in y_1+\cdots+y_j+\Sigma_{k-j}(B)\subseteq\Sigma_k(A)$. Thus
$\Sigma_k(A)=G$ for $s\le k\le m/2$, and \cref{lem:symmetry} gives the
full interval.
\end{proof}

\Needspace{12\baselineskip}
\section{Asymptotic critical density}\label{sec:asymptotic-density}

The lower bound in \cref{thm:asymptotic-critical} is immediate from Bajnok's theorem. Indeed, there exists a set $B\subset G$ of size $\alpha_k(N)N$ with $kB\ne G$, and hence also $\Sigma_k(B)\ne G$. We now prove that this lower bound is asymptotically sharp for restricted sums.

\begin{proof}[Proof of \cref{thm:asymptotic-critical}]
Fix $k$ and $\eps$ as in the statement, and put $\eta=\eps/(2k)$.  Let $\rho_k(\eta)$ be supplied by
\cref{lem:removal}, and choose $N_0=N_0(k,\eps)$ so that
$N_0\ge \binom{k}{2}/\rho_k(\eta)$.
Let $N\ge N_0$, and suppose for a contradiction that
$g\notin\Sigma_k(A)$ for some $g\in G$.  Since $3\le k\le2p-3$, we have
$p\ge3$, so $G$ has odd order.

Every ordinary ordered representation
$a_1+\cdots+a_k=g$, with $a_i\in A$, has a collision.  For a prescribed pair $i<j$ with $a_i=a_j$, choosing the
other $k-2$ entries determines the repeated entry uniquely from
$2a_i=g-\sum_{t\ne i,j}a_t$, because doubling is a bijection on $G$.  Hence the number of ordinary
representations is at most
\[
\binom{k}{2}N^{k-2}\le \rho_k(\eta)N^{k-1}.
\]
Apply \cref{lem:removal} to the $k$ copies of $A$.  There are subsets
$A_1,\ldots,A_k\subseteq A$ such that
$|A_i|\ge |A|-\eta N$ for $1\le i\le k$, and
$g\notin A_1+\cdots+A_k$.
Set $B=A_1\cap\cdots\cap A_k$.  Then
\[
|B|\ge |A|-k\eta N
\ge \left(\alpha_k(N)+\frac{\eps}{2}\right)N
>\alpha_k(N)N.
\]
Since $\alpha_k(N)N$ is an integer, $|B|\ge\alpha_k(N)N+1$.
By Bajnok's formula \cite[Theorem~9]{BajnokCritical}, this forces
$kB=G$.  In particular $g\in kB\subseteq A_1+\cdots+A_k$, a contradiction.
Thus $\Sigma_k(A)=G$.
\end{proof}

\begin{proof}[Proof of \cref{cor:positive-density}]
Put $K=\floor{1/c}+1$, $L=\floor{2/c}+1$, and
$K_0=\max\{3,K\}$.
Then $jc>1$ for $K_0\le j\le L$, and $L(c/2)>1$.  Choose $\delta>0$ so small that $1/j+2\delta<c$ for
$K_0\le j\le L$ and $1/L+2\delta<c/2$.
For $K_0\le j\le L$, let $N_0(j,\delta)$ be an order threshold supplied
by \cref{thm:asymptotic-critical}.  Choose $p_0(c)$ so large that
\[
p_0(c)>\max\left\{\frac1\delta,\frac{L+3}{2},
\max_{K_0\le j\le L}N_0(j,\delta)\right\}.
\]
If $c>1/2$, enlarge $p_0(c)$ further so that $p_0(c)>1/(2c-1)$.

Let $N=|G|$ and assume $p(G)>p_0(c)$.  Every divisor $d>1$ of $N$ satisfies
$d\ge p(G)$, and therefore, for every $j\ge1$,
\begin{equation}\label{eq:alpha-upper}
\alpha_j(N)
=\max_{\substack{d\mid N\\ d>1}}\frac{\ceil{(d-1)/j}}{d}
\le \frac1j+\frac1{p(G)}.
\end{equation}
Indeed, $\ceil{x}\le x+1$ gives
$\ceil{(d-1)/j}/d<1/j+1/d$.

For $K_0\le j\le L$, \eqref{eq:alpha-upper} and the choice of $p_0(c)$ give
$\alpha_j(N)+\delta<c$.
Also $j\le L<2p(G)-3$.  Hence \cref{thm:asymptotic-critical}, applied with
error $\delta$, yields $\Sigma_j(A)=G$.  If $c>1/2$, then $K=2$.
The inequality $N\ge p(G)>1/(2c-1)$ gives $cN>(N+1)/2$, so the odd-order
two-fold covering result \cite[Lemma~2.2]{GallardoEtAl} gives
$\Sigma_2(A)=G$.  Thus all lengths $K\le j\le L$ are covered.

Now let $L\le k\le |A|/2$ and choose
$F\subset A$ with $|F|=k-L$.  Then $|A\setminus F|\ge |A|/2\ge cN/2$. By
\eqref{eq:alpha-upper}, $\alpha_L(N)+\delta<c/2$, so another application of
\cref{thm:asymptotic-critical} gives $\Sigma_L(A\setminus F)=G$. Therefore
$\sigma_F+\Sigma_L(A\setminus F)=G\subseteq\Sigma_k(A)$.
The filler $F$ is disjoint from $A\setminus F$. This proves every lower-half length $K\le k\le |A|/2$, and \cref{lem:symmetry} gives the upper half. Since $k$ is an integer, the resulting range is exactly $1/c<k<|A|-1/c$.

It remains to prove that the lower endpoint is best possible. Let $N$ be a sufficiently large prime and put $A=\{0,\ldots,M-1\}\subset C_N$, where $M=\lceil cN\rceil$. For every integer $1\le k\le1/c$, each restricted $k$-sum is the reduction modulo $N$ of an integer in $[\binom{k}{2},\,k(M-1)-\binom{k}{2}]$. Thus there are at most $k(M-k)+1$ possible sums. For $k=1$ this is $M<N$ once $N$ is large. For $k\ge2$, using $M<cN+1$ and $kc\le1$ gives
\[
k(M-k)+1<kcN-k(k-1)+1\le N-k(k-1)+1<N.
\]
Hence $\Sigma_k(A)\ne C_N$ for every integer $k\le1/c$, while $p(C_N)=N$ can be arbitrarily large. Hence the lower endpoint is best possible.
\end{proof}

\section{Application to MDS elliptic codes}\label{sec:elliptic}
Let $E/\F_q$ have identity $O$, put $G=E(\F_q)$ and $N=|G|$, and let
$P=\{P_1,\ldots,P_n\}\subset G$. For an $\F_q$-rational divisor $D$ whose
support is disjoint from $P$, the elliptic evaluation code is
\[
C_L(P,D)=\{(f(P_1),\ldots,f(P_n)):f\in\mathcal L(D)\}.
\]
For $0<k=\deg D<n$, evaluation is injective, and Riemann--Roch gives
$\dim C_L(P,D)=k$ \cite[Theorem~1.5.15]{Stichtenoth}. Let
$\kappa(D)\in E(\F_q)$ correspond to the class of $D-k(O)$ under
$\operatorname{Pic}^0(E)\cong E$.

\begin{lemma}\label{lem:MDS-criterion}
The code $C_L(P,D)$ is not MDS if and only if $\kappa(D)\in\Sigma_k(P)$.
\end{lemma}
This is Han--Ren \cite[Lemma~3.4]{HanRen}.

When $|E(\F_q)|$ is even, Roth and Lempel's theorem \cite[Theorem~3.1 and Remark~3.1]{RothLempel} applies directly once $N>605$, since $E(\F_q)$ has rank at most two \cite[Theorem~3]{Ruck} and hence none of the exceptional cases can occur.

\begin{proof}[Proof of \cref{thm:elliptic}]
Put $G=E(\F_q)$ and $N=|G|$, and represent $C$ as $C_L(P,D)$, with $|P|=n$ and $\deg D=k$. The Hasse--Weil bound \cite[Theorem~5.2.3]{Stichtenoth} gives $N\ge(\sqrt q-1)^2$.

For part~\textup{(a)}, $q>655$ implies $N>605$. If $n>N/3+1$, then \cref{thm:sharp-density} gives $\Sigma_k(P)=G$, contradicting \cref{lem:MDS-criterion}. Hence $n\le N/3+1$.

For part~\textup{(b)}, let $p$ be the least prime divisor of $N$. For fixed $p$, sufficiently large $q$ makes $N$ exceed the threshold $N_0(p)$ in \cref{thm:least-prime}. If $n>N/p+p-2$, then \cref{thm:least-prime} gives $\Sigma_k(P)=G$, again contradicting \cref{lem:MDS-criterion}. Hence $n\le N/p+p-2$.
\end{proof}

\begin{proof}[Proof of \cref{cor:elliptic-half}]
Put $G=E(\F_q)$ and $N=|G|$, and represent $C$ as $C_L(P,D)$, with $|P|=n$ and $\deg D=k$. The Hasse--Weil bound \cite[Theorem~5.2.3]{Stichtenoth} gives $N\ge q+1-2\sqrt q=(\sqrt q-1)^2$. For $q>655$, the right-hand side is greater than $605$, so $N>605$.

Suppose first that $N$ is odd. If $k\in\{3,n-3\}$ and $n>2N/5$, then \cref{prop:three-fold}(a), together with \cref{lem:symmetry} when $k=n-3$, gives $\Sigma_k(P)=G$, contradicting \cref{lem:MDS-criterion}. Hence $n\le2N/5<N/2$. If $4\le k\le n-4$, \cref{thm:elliptic}(a) gives $n\le N/3+1<N/2$. This proves the result for all $3\le k\le n-3$.

Now suppose that $N$ is even and assume for a contradiction that $n>N/2$. Since $E(\F_q)$ has rank at most two \cite[Theorem~3]{Ruck} and $N>605$, $G$ is not one of the exceptional endpoint families in \cref{thm:roth-lempel}. Hence part~\textup{(a)} of that theorem covers $k\in\{3,n-3\}$, and part~\textup{(b)} covers every $4\le k\le n-4$. Thus $\Sigma_k(P)=G$ throughout $3\le k\le n-3$, contradicting \cref{lem:MDS-criterion}. Hence $n\le N/2$.
\end{proof}

\begin{remark}\label{rem:elliptic-sharp}
The half-scale bound is sharp. For every $q>7$, Waterhouse's theorem \cite[Theorem~4.1(1)]{Waterhouse} provides an elliptic curve $E/\F_q$ with $N=|E(\F_q)|$ even, and Han--Ren constructed an MDS elliptic code of length $N/2$; see the proof of Theorem~1 in \cite{HanRenTIT} and \cite[Section~3.4]{HanRen}.

The least-prime bound in \cref{thm:elliptic}(b) is also exact. Fix an odd prime $p$, and let
$M=\prod_{\ell\le p,\ \ell\ {\rm prime}}\ell$. Every positive integer
$N\equiv p\pmod M$ has least prime divisor $p$. Since the Hasse interval
$[q+1-2\sqrt q,\,q+1+2\sqrt q]$ has length $4\sqrt q$, for every sufficiently
large prime $q$ it contains at least two integers congruent to $p$ modulo $M$;
choose one of them, say $N$, with $N\ne q+1$. Then
$\beta=q+1-N$ satisfies $|\beta|\le2\sqrt q$ and $(\beta,q)=1$, so
Waterhouse's theorem \cite[Theorem~4.1(1)]{Waterhouse} gives an elliptic
curve $E/\F_q$ with $|E(\F_q)|=N$. Let $P$ be the index-$p$ construction
preceding \cref{thm:least-prime}. Then $|P|=N/p+p-2$, and one quotient
class is disjoint from $\Sigma_j(P)$ for every $0\le j\le|P|$. Fix a point
$Q$ in this missing quotient class and a degree $1\le j<|P|$. The divisor
$D_0=(j-1)(O)+(Q)$ has degree $j$ and $\kappa(D_0)=Q$. By the strong
approximation theorem \cite[Theorem~1.6.5]{Stichtenoth}, $D_0$ is linearly
equivalent over $\F_q$ to a divisor $D$ whose support is disjoint from $P$.
Then $\kappa(D)=Q\notin\Sigma_j(P)$, so \cref{lem:MDS-criterion} shows that
$C_L(P,D)$ is an MDS elliptic code of length $N/p+p-2$.
\end{remark}

\begingroup
\small\raggedright

\endgroup

\end{document}